\pdfoutput=1
\documentclass[12pt,reqno]{amsart}
\usepackage[letterpaper,margin=1in,footskip=0.25in]{geometry}
\usepackage{mathrsfs}
\usepackage{amssymb}
\usepackage{mathtools}
\usepackage{tikz}
\usepackage{tikz-cd}
\usepackage{setspace}
\usepackage{enumitem}
\usepackage[backend=biber, style=alphabetic]{biblatex}
\PassOptionsToPackage{pdfusetitle,colorlinks}{hyperref}
\usepackage{bookmark}
\hypersetup{
  linkcolor={blue!70!black},
  citecolor={blue!70!black},
  urlcolor={teal!90!black}
}

\newtheorem{theorem}{Theorem}[section]
\newtheorem{lemma}[theorem]{Lemma}
\newtheorem{proposition}[theorem]{Proposition}
\newtheorem{corollary}[theorem]{Corollary}

\theoremstyle{definition}

\theoremstyle{remark}
\newtheorem{remark}[theorem]{Remark}
\newtheorem{example}[theorem]{Example}

\newcommand{\sC}{\mathcal{C}}

\newcommand{\sE}{\mathcal{E}}
\newcommand{\sF}{\mathcal{F}}

\newcommand{\sH}{\mathcal{H}}

\newcommand{\sL}{\mathcal{L}}
\newcommand{\sM}{\mathcal{M}}
\newcommand{\sN}{\mathcal{N}}
\newcommand{\sO}{\mathcal{O}}

\newcommand{\sX}{\mathcal{X}}
\newcommand{\sY}{\mathcal{Y}}
\newcommand{\sZ}{\mathcal{Z}}

\newcommand{\bbA}{\mathbb{A}}

\newcommand{\bbC}{\mathbb{C}}

\newcommand{\bbG}{\mathbb{G}}

\newcommand{\bbN}{\mathbb{N}}
\newcommand{\bbP}{\mathbb{P}}
\newcommand{\bbQ}{\mathbb{Q}}
\newcommand{\bbR}{\mathbb{R}}

\newcommand{\bbV}{\mathbb{V}}
\newcommand{\bbZ}{\mathbb{Z}}

\newcommand{\onto}{\rightarrow\hspace*{-.14in}\rightarrow}

\renewcommand{\div}{\text{div}}

\DeclareMathOperator{\Pic}{Pic}
\DeclareMathOperator{\coker}{coker}
\DeclareMathOperator{\im}{im}

\DeclareMathOperator{\id}{id}
\DeclareMathOperator{\Hom}{Hom}

\DeclareMathOperator{\Spec}{Spec}

\DeclareMathOperator{\Jac}{Jac}
\DeclareMathOperator{\Aut}{Aut}

\DeclareMathOperator{\Char}{Char}

\DeclareMathOperator{\Fix}{Fix}
\DeclareMathOperator{\Gal}{Gal}

\DeclareMathOperator{\br}{Br}
\DeclareMathOperator{\sing}{sing}
\DeclareMathOperator{\reg}{reg}
\DeclareMathOperator{\tor}{tor}
\DeclareMathOperator{\sm}{sm}

\DeclareMathOperator{\NS}{NS}

\begin{document}
\title[On the Brauer group of a generic Godeaux surface]{On the Brauer group of a generic Godeaux surface}

\author{Theodosis Alexandrou}
\setstretch{1.16}

\subjclass[2010]{}

\makeatletter
  \hypersetup{
    pdfauthor=,
    pdfsubject=\@subjclass,
    pdfkeywords=\@keywords
  }
\makeatother

\maketitle
\begin{abstract} Let $X$ be a Godeaux surface and $q_{X}\colon Y\to X$ be its universal cover. We show that the pullback map $q_{X}^{*}\colon\br(X)\to\br(Y)$ is injective if $\rho(Y)=9$. Our arguments rely on a degeneration technique that also applies to other examples.\end{abstract}
\section{Introduction}\label{sec:1}The Brauer group $\br(X)$ of a smooth projective variety $X$ over a field $k$ is the abelian group $H^{2}(X_{\text{\'et}},\bbG_{m})$. It can be also regarded as the group of equivalence classes of \'etale locally trivial $\bbP^{n}$ bundles over $X$ modulo Zariski locally trivial ones. This is always a torsion group and an important birational invariant of $X$.\par Let $q\colon Y\to X$ be a finite \'etale covering. This defines a natural pull-back map $q^{*}\colon \br(X)\to\br(Y)$ on Brauer groups. Since $q_{*}q^{*}=\deg(q)\cdot\id$, the kernel of $q^{*}$ is killed by the positive integer $d\coloneqq\deg(q)$. Therefore, detecting whether a given $d$-torsion class $a\in\br(X)[d]$ pulls back to zero is a subtle problem.\par Beauville studied this question in the case of complex Enriques surfaces, which was first stated in \cite{hsk}. Namely, an Enriques surface $X$ admits a natural double cover $Y$, which is a $K3$ surface. Beauville \cite{bea} showed that the kernel of the map $\bbZ/2=\br(X)\to\br(Y)=(\bbQ/\bbZ)^{22-\rho},\ 1\leq\rho\leq 20,$ depends on $X$. More specifically, he described lattice theoretically the surfaces $X$ in the coarse moduli space of Enriques for which the kernel of $q^{*}_{X}$ is non-trivial and showed that this locus is a countable, infinite union of non-empty algebraic hypersurfaces. His method fails to give a concrete example but among others Garbagnati and Sch\"utt \cite{gm} constructed examples of Enriques surfaces $X$ over $\bbQ$ with $q^{*}_{X}$ being injective or trivial.\par Recent work of Bergstr\"om, Ferrari, Tirabassi and Vodrup \cite{bie} uses Beauvilles' method to answer similar questions for a Bi-elliptic surface $X$ over $\bbC$.\par In \cite{talk}, Beauville asked the above question for Godeaux surfaces. A Godeaux surface $X$ is a classical example of a surface of general type with $H^{i}(X,\sO_{X})=0$ for $i=1,2$ and $\NS(X)_{\tor}\neq 0$($\cong\bbZ/5$). They usually appear as $\bbZ/5$-quotients of smooth quintics $Y\subset\bbP^{3}_{\bbC}$ that are invariant and fixed point free by the automorphism \[\varphi\colon (x_{0}:x_{1}:x_{2}:x_{3})\mapsto (x_{0}:\zeta x_{1}:\zeta^{2}x_{2}:\zeta^{3}x_{3}),\tag{1.1}\label{1.1}\]where $\zeta$ is a primitive $5$-th root of unity (see \cite[$\S$2]{lw}).\par The main result of this paper answers the question for $\rho(Y)=9$, which is the generic case (see \cite[Example 3]{BoM}).\begin{theorem}\label{thm:1.1}Let $X$ be a Godeaux surface over $\bbC$ and $q_{X}\colon Y\to X$ be its universal cover. If the Picard number of $Y$ is $9$, then the pull-back map $q^{*}_{X}\colon\br(X)\to\br(Y)$ is injective.\end{theorem} As an immediate consequence, we also obtain:\begin{corollary}\label{cor:1.2}For the very general Godeaux surface $X$ over $\bbC$, the pull-back map $q^{*}_{X}$ is injective.\end{corollary}Theorem \ref{thm:1.1} will be deduced from the following result.\begin{theorem}\label{thm:1.3} Let $k$ be an algebraically closed field of characteristic $0$. Pick a generic pencil $\sY\to\bbP^{1}_{k}\subset|\sO_{\bbP^{3}_{k}}(5)|^{\bbZ/5}$ of $\bbZ/5$-invariant quintics, with some fibre $\sY_{t_{0}},\ t_{0}\in\bbP^{1}_{k}(k)$ not passing through any of the 4 fixed points of the action \eqref{1.1} and having 5 triple points as its only singularities. Let $q\colon\sY\to\sX$ be the quotient map. Then the pull-back map $q^{*}_{\bar{\eta}}\colon\br(\sX_{\bar{\eta}})\to\br(\sY_{\bar{\eta}})$ is injective.\end{theorem} A pencil with the properties mentioned in Theorem \ref{thm:1.3} exists. This relies on the fact that the invariant quintics $|\sO_{\bbP^{3}_{k}}(5)|^{\bbZ/5}$ form a 12-dimensional linear space and for a given point $p\in\bbP^{3}_{k}(k)$ not in any coordinate hyperplane, we can always find a quintic surface $Y$ invariant and fixed point free under the $\bbZ/5$-action \eqref{1.1} with 5 triple points at the orbit of $p$ and no other singularities (see \cite[Appendix 5, Lemma 1]{persson}).\par It remains unclear whether there are actually examples of Godeaux surfaces $X$, such that $q^{*}_{X}=0$.\par Our method is adjustable and can be also applied to other surfaces. In section \ref{sec:6} we show that complex Enriques surfaces $X$ with $q^{*}_{X}=0$ do not admit a type II degeneration (see Theorem \ref{thm:6.2}). By comparing the latter with Beauvilles' result \cite[Corollary 6.5]{bea}, one can determine hypersurfaces in the coarse moduli space of Enriques, with the property that any of their members cannot be degenerated to a type II (see Corollary \ref{cor:6.5}). In the last section we also study the case of cyclic quotients of products of curves (see Theorem \ref{thm:7.1}).\subsection{A degeneration technique.}\label{subsec:1.1} Theorem \ref{thm:1.3} relies on the following, which is the crucial technical result of the paper.\begin{theorem}\label{thm:1.4} Let $q\colon\sY\to\sX$ be a finite \'etale Galois covering of strictly semi-stable $R$-schemes, such that the \'etale cover $q_{0}$ is trivial over $\sX^{\sing}_{0}\cap\sX_{0,i}$ for every component $\sX_{0,i}$ of the special fibre $\sX_{0}$. Assume that the degree $d\coloneqq\deg(q)$ is invertible in $R$. Then \eqref{1'''} implies \eqref{2'''}:\begin{enumerate}
    \item \label{1'''} The pull-back $q_{\bar{\eta}}^{*}\colon\br(\sX_{\bar{\eta}})[d]\to\br(\sY_{\bar{\eta}})[d]$ is the zero map.\item\label{2'''} Up to some finite ramified base change of discrete valuation rings $R\subset\tilde{R}$ followed by a resolution of $\sX_{\tilde{R}},$ the restriction $\br(\sX)[d]\to\br(\sX_{\bar{\eta}})[d]$ is surjective.  
\end{enumerate}\end{theorem} When performing a ramified base change $R\subset\tilde{R},$ the model $\sX_{\tilde{R}}$ becomes singular, but it follows from \cite[Proposition 2.2]{Har}, that the family $\sX_{\tilde{R}}\to\Spec\tilde{R}$ can be made again into a strictly semi-stable by repeatedly blowing up the non-Cartier divisors of the special fibre.\par Theorem \ref{thm:1.4} establishes a link between the following two properties (see also Theorem \ref{thm:3.1}):\begin{enumerate}[label=(\roman*)] 
    \item\label{i} A class $\alpha\in\br(\sX_{\bar{\eta}})$ pulls back to zero, i.e. $q_{\bar{\eta}}^{*}\alpha=0\in\br(\sY_{\bar{\eta}})$.\item\label{ii} Up to some finite base change of discrete valuation rings $R\subset\tilde{R}$ followed by a resolution of $\sX_{\tilde{R}}$, $\alpha\in\br(\sX_{\bar{\eta}})$ lifts to an honest class in $\br(\sX)$.\end{enumerate}      
The advantage of relating these properties is that in many cases it is not hard to check whether a given class $\alpha\in\br(\sX_{\bar{\eta}})$ lifts to the total space $\sX$. In all applications presented in this paper, the restriction map $\br(\sX)\to\br(\sX_{\bar{\eta}})$ happens to be trivial, either because $\br(\sX)=0$ or $\br(\sX)$ is a divisible group and $\br(\sX_{\bar{\eta}})$ is finite (see Lemmas \ref{lem:5.2}, \ref{lem:6.3} and \ref{lem:7.2}). An example with $\br(\sX)\to\br(\sX_{\bar{\eta}})$ being surjective appears in \cite[Theorem 7.1]{Stef}.\par The property \ref{ii} had been also studied in \cite{Stef}, where a relation to algebraic torsion classes had been found.\subsection{Outline of the proofs.}\label{subsec:1.2} The proof of Theorem \ref{thm:1.4} relies on the following observations. Pick a Brauer class $\alpha\in\br(\sX_{\eta})\{\ell\}$ with $q^{*}_{\eta}\alpha=0$. Up to a suitable finite ramified base change, it is possible to achieve vanishing of the residues $\partial_{i}\alpha\in H^{1}_{\text{\'et}}(\sX_{0,i},\bbQ_{\ell}/\bbZ_{\ell})$ for all irreducible components $\sX_{0,i}$ that have not been introduced by resolving the singularities that appeared due to the base change. Using our assumption for $q_{0},$ we see that over the new components the covering map $q$ is trivial. Hence, the usual pull-push argument shows that the remaining residues of $\alpha$ coincide with the ones of its pull-back $0=q^{*}_{\eta}\alpha\in\br(\sY_{\eta})\{\ell\}$ and so, all vanish. This will conclude the proof of Theorem \ref{thm:1.4}.\par Let $\sY\to\bbP^{1}_{k}\subset|\sO_{\bbP^{1}_{k}}(5)|^{\bbZ/5}$ be as in Theorem \ref{thm:1.3} and perform a base change with respect to the local ring at $t_{0}\in\bbP^{1}_{k}$. By applying semi-stable reduction to $\sY$ and its $\bbZ/5$-quotient $\sX$, we arrive at a $5:1$ \'etale cover $q\colon\sY\to\sX$, that is trivial over $\sX^{\sing}_{0}\cap\sX_{0,i}$ for all components $\sX_{0,i}$ of the special fibre $\sX_{0}$. Replacing $R\coloneqq\sO_{\bbP^{1}_{k},t_{0}}$ with its completion $\hat{R}\cong k[[t]]$, the proper base change theorem can be used to show that the restriction map $\br(\sX)\to\br(\sX_{0})_{\tor}$ is an isomorphism (see Lemma \ref{Lem:4.1}). Then an explicit calculation using the Mayer-Vietoris exact sequence gives $\br(\sX_{0})=0$ (see Lemma \ref{lem:5.2}). Hence, $\br(\sX)=0$ and so a Brauer class $0\neq\alpha\in\br(X)\cong\bbZ/5$ never lifts to the total space $\sX$. We thus obtain Theorem \ref{thm:1.3} by applying Theorem \ref{thm:1.4}.\par Theorem \ref{thm:1.3} can be used to produce a Godeaux surface $X$ with universal cover $Y$ of Picard number $9$, such that $q^{*}_{X}$ is injective. As a final step, a specialization argument comes in hand, which shows that the map in question is in fact injective for every $Y$ with $\rho(Y)=9$ (see Proposition \ref{prp:4.4}). This completes the proof of Theorem \ref{thm:1.1}.\section{Preliminaries}\label{sec:2}\subsection{Conventions and notations.}\label{subsec:2.1} Let $A$ be an abelian group and let $\ell$ be a prime number. We denote by $A[\ell^{r}]$ the subgroup of $\ell^{r}$-torsion elements. We write $A\{\ell\}$ for the $\ell$-primary subgroup of $A$. If $n\in\bbN$, then $[n]\colon A\to A$ is the map that sends $x$ to $nx$. Also $A_{\div}$ denotes the maximal divisible subgroup of $A$.\par All schemes are assumed to be seperated. A variety is a geometrically integral scheme of finite type over a field. The regular locus of $X$ will be denoted by $X^{\reg}$ and its singular by $X^{\sing}$. For an $S$-scheme $X\to S$ and any morphism $S'\to S$, we denote by $X_{S'}\coloneqq X\times_{S}S'$ the base change.\par We denote by $k^{sep}$ and $\bar{k}$ the seperable and algebraic closure of a field $k$, respectively. We use the notation $R^{\text{h}}$ for the henselization of a local ring $R$.\par Let $R$ be a discrete valuation ring with residue field $k$ and fraction field $K$. For any $R$-scheme $\sX\to\Spec R$, we write $\sX_{0}\coloneqq\sX\times_{R}k$ (resp. $\sX_{\bar{0}}\coloneqq\sX\times_{R}\bar{k}$) for the special (resp. geometric special) fibre and $\sX_{\eta}\coloneqq\sX\times_{R}K$ (resp. $\sX_{\bar{\eta}}\coloneqq\sX\times_{R}\bar{K}$) for the generic (resp. geometric generic) fibre.\par A projective flat $R$-scheme $\sX\to\Spec R$ is called \textit{strictly semi-stable}, if $\sX$ is an integral regular scheme, the generic fibre $\sX_{\eta}$ is smooth and the special fibre $\sX_{0}$ is a geometrically reduced simple normal crossing divisor on $\sX$, i.e. the irreducible components $\sX_{0}$ are all smooth varieties and the scheme-theoretic intersection of $n$ distinct components is either empty or smooth and equi-dimensional of codimension $n$ in $\sX$. A strictly semi-stable $R$-scheme $\sX\to\Spec R$ is called \textit{triple-point free} if the intersection of any three pairwise distinct components in the special fibre is empty.\subsection{\'Etale Cohomology and Brauer Group.}\label{subsec:2.2} We collect some results from \'etale cohomology that will be used through out this paper. The results that we discuss can be found in \cite{mil} and \cite{cts}. For a scheme $X$, we denote by $X_{\text{\'et}}$ the small \'etale site of $X$. If $\sF\in Sh(X_{\text{\'et}})$ is a sheaf of abelian groups, then we write $H^{i}(X,\sF)\coloneqq H^{i}(X_{\text{\'et}},\sF)$ for the \'etale cohomology groups with coefficients in $\sF$. If $X=\Spec A$ for some ring $A$, we prefer to write $H^{i}(A,F)$, instead. For a positive integer $n$, we denote by $\mu_{n}$ the subsheaf of $n$-th roots of unity of the multiplicative sheaf $\bbG_{m}$ on $X_{\text{\'et}}$.\par Let $X$ be a scheme and $n$ a positive integer that is invertible on $X$. We shall frequently use the following description of the $n$-torsion subgroup of the Brauer group of $X$, which is induced from the Kummer sequence \[\coker(c_{1}\colon\Pic(X)\to H^{2}(X,\mu_{n}))\cong \br(X)[n].\tag{2.1}\label{2.6}\]\begin{lemma}\label{lem:2.1}Let $X$ be an integral, regular scheme of finite type over a field $k$ (resp. a dvr $R$) and let $Z\subset X$ be a closed subset of codimension one in $X$. Let $U\subset X$ be the complement of $Z$ in $X$. Let $\ell$ be a prime invertible on $X$. If $Z_{1}, Z_{2},\ldots,Z_{t}$ are the components of $Z$ with codimension one in $X$, then we have the following exact sequences:\[0\to\br(X)[\ell^{r}]\to\br(U)[\ell^{r}]\overset{\{\partial_{i}\}}\to\bigoplus_{i=1}^{t}H^{1}(k(Z_{i}),\bbZ/\ell^{r}),\tag{2.2}\label{2.1}\]\[0\to\br(X)\{\ell\}\to\br(U)\{\ell\}\overset{\{\partial_{i}\}}\to\bigoplus_{i=1}^{t}H^{1}(k(Z_{i}),\bbQ_{\ell}/\bbZ_{\ell}).\tag{2.3}\label{2.2}\]\end{lemma}\begin{proof} See \cite[Section\ 3.7]{cts} for a detailed proof.\end{proof}\begin{lemma}\label{lem:2.2} Let $f\colon X'\to X$ be a flat morphism of relative dimension $n\coloneqq\dim X'-\dim X$, between integral, regular schemes of finite type over a field $k$ (resp. a dvr $R$). For a prime divisor $Z\subset X$, we write $f^{*}Z=\sum_{i=1}^{t}\lambda_{i}Z'_{i}$ for its pull-back, where $Z'_{i}$ are the prime divisors on $X'$ supported on $f^{-1}(Z)$ and $\lambda_{i}$ denote their multiplicities. Set $U\coloneqq X\setminus Z$ and $U'\coloneqq f^{-1}(U)\subset X'$. Then for any prime $\ell$ invertible on $X$, the following diagram commutes:\[\begin{tikzcd}
{\br(U)[\ell^{r}]} \arrow[r, "\partial"] \arrow[d, "f^{*}"'] & {H^{1}(k(Z),\bbZ/\ell^{r})} \arrow[d, "{\{\lambda_{i}f|_{Z'_{i}}^{*}\}_{i}}"'] \\
{\br(U')[\ell^{r}]} \arrow[r, "\{\partial_{i}\}"]                    & {{\bigoplus_{i=1}^{t}H^{1}(k(Z'_{i}),\mathbb{Z}/\ell^{r})}.}          
\end{tikzcd}\tag{2.4}\label{2.3}\]\end{lemma}
\begin{proof} This follows from \cite[Theorem 3.7.5]{cts}.\end{proof}
\begin{lemma}\label{lem:2.3} Let $R$ be a discrete valuation ring with residue field $k$ and fraction field $K$. Let $\sX\to\Spec R$ be a triple-point free strictly semi-stable $R$-family. Then we have the following Mayer-Vietoris exact sequence:\[\dots\to\bigoplus_{i}\Pic(\sX_{0,i})\overset{r_{1}}\to\bigoplus_{i<j}\Pic(\sX_{0,i}\cap\sX_{0,j})\to\br(\sX_{0})\overset{r_{2}}\to\bigoplus_{i}\br(\sX_{0,i})\to \dots,\tag{2.5}\label{2.4}\] where the map $r_{1}$ sends $(\sL_{i})$ to $(\sL_{i}|_{\sX_{0,i}\cap\sX_{0,j}}-\sL_{j}|_{\sX_{0,i}\cap\sX_{0,j}})$ and $r_{2}$ takes $\alpha$ to $(\alpha|_{\sX_{0,i}})$.\end{lemma}\begin{proof} Consider the morphisms $p\colon Y\coloneqq\bigsqcup_{i}\sX_{0,i}\to\sX_{0}$ and $\rho_{i,j}\colon X_{i,j}\coloneqq\sX_{0,i}\cap\sX_{0,j}\hookrightarrow\sX_{0}$ for $i\neq j$. The triple point free assumption implies exactness of the following sequence of sheaves for the \'etale topology on $\sX_{0}$:\[1\to\bbG_{m,\sX_{0}}\to p_{*}\bbG_{m,Y}\to \bigoplus_{i<j}{\rho_{i,j}}_{*}\bbG_{m,X_{i,j}}\to 1. \tag{2.6}\label{2.5}\] Since the morphisms $p$ and $\rho_{i,j}$ are finite, the push-forwards $p_{*}$ and ${\rho_{i,j}}_{*}$ are both exact functors in the \'etale topology (see \cite[Corollary II.3.6]{mil}). Hence, $R^{s}p_{*}=R^{s}{\rho_{i,j}}_{*}=0$ for any $s\geq 1$ and the Leray spectral sequence gives rise to the isomorphisms $H^{s}(\sX_{0},p_{*}\bbG_{m})\cong H^{s}(Y,\bbG_{m})$ and  $H^{s}(\sX_{0},{\rho_{i,j}}_{*}\bbG_{m})\cong H^{s}(X_{i,j},\bbG_{m})$. Finally, the associated long exact sequence of \eqref{2.5} together with the use of the above isomorphisms lead to the sequence \eqref{2.4}.\end{proof}\section{Main Technical Result}\label{sec:3} Theorem \ref{thm:1.4} is an immediate corollary of the following.\begin{theorem}\label{thm:3.1} Let $q\colon\sY\to\sX$ be a finite \'etale Galois covering of strictly semi-stable $R$-schemes, such that the \'etale cover $q_{0}$ is trivial over $\sX^{\sing}_{0}\cap\sX_{0,i}$ for every irreducible component $\sX_{0,i}$ of $\sX_{0}$. Fix a prime number $\ell$, that is invertible in $R$.  Pick a class $\alpha\in\br(\sX_{\bar{\eta}})\{\ell\}$. Then the following are equivalent:\begin{enumerate}\item\label{1} The class $\alpha$ pulls-back to zero, i.e. $q_{\bar{\eta}}^{*}\alpha=0$.\item\label{2} Up to some finite ramified base change of discrete valuation rings $R\subset\tilde{R}$ followed by a resolution of $\sX_{\tilde{R}},$ the class $\alpha$ lifts to a class $\tilde{\alpha}$ in $\br(\sX)\{\ell\}$ with $q^{*}(\tilde{\alpha})=0$.\end{enumerate}\end{theorem}\begin{proof} The implication \eqref{2} $\implies$ \eqref{1} is clear by functoriality of pullbacks.\par We prove the implication \eqref{1} $\implies$ \eqref{2}. We may assume that $\alpha\in\br(\sX_{\bar{\eta}})[\ell^{r}]$ for some $r\geq 1$, such that $\ell^{r}$ divides $d\coloneqq\deg(q)$. By \cite[Section 2.2.2]{cts}, we have \[\br(\sX_{\bar{\eta}})=\lim_{\underset{K\subset L}\longrightarrow}\br(\sX_{L}),\] where $\sX_{L}=\sX_{\eta}\times_{K} L$ and the co-limit is taken over all finite seperable field extensions $L$ of $K$. Hence, we can find a finite seperable field extension $L$ of $K$, such that $\alpha\in\im(\br(\sX_{L})[\ell^{r}]\to\br(\sX_{\bar{\eta}})[\ell^{r}])$. We pick a lift of $\alpha$ in $\br(\sX_{L})[\ell^{r}]$ and denote it again by the same symbol. Since $q^{*}_{\bar{\eta}}\alpha=0,$ the same reasoning as above gives $q^{*}_{L}\alpha=0$ (up to replacing $L$ if neccessary with some finite seperable field extension). We let $\tilde{R}\subset L$ be the normalization of $R$ in $L$. Then $\tilde{R}$ is a Dedekind domain finite over $R$ (see \cite[Proposition 12.53]{gw}). Choose a maximal ideal $\mathfrak{m}\subset R$ and consider the discrete valuation ring $R'\coloneqq \tilde{R}_{\mathfrak{m}}$. We perform a base change corresponding to the extension of discrete valuation rings $R'$ of $R$ and note that the models $\sX$ and $\sY$ may become singular. By \cite[Proposition 2.2]{Har}, both $\sX_{R'}$ and $\sY_{R'}$ can be made into strictly semi-stable $R'$-schemes. Specifically, by repeatedly blowing up all non-Cartier components of the special fibre of $\sY\to\Spec R'$, we can arrive at a strictly semi-stable model $\sY'\to\Spec R'$. As we blow up along $G$-invariant centers, the action on $\sY$ given by the Galois group $G\coloneqq\Aut(q)$, can be naturally lifted to a fixed point free action on $\sY'$. Passing to the quotient $\sX'\coloneqq\sY'/G$, we obtain a semi-stable model for $\sX_{R'}$. In particular, according to the above, we may assume that $\alpha\in\br(\sX_{\eta})[\ell^{r}]$ and $q_{\eta}^{*}\alpha=0\in\br(\sY_{\eta})[\ell^{r}]$.\par Fix a uniformizer $\pi\in R$ and perform the $\ell^{r}:1$ base change $\pi\mapsto\pi^{\ell^{r}}$ followed by resolving the singularities, as above. We keep the notations $\tilde{\sX},\tilde{\sY}$ and $\tilde{q}\colon\tilde{\sY}\to\tilde{\sX}$ for the resulting models and covering over $R$. Write $\sX_{0}=\sum_{t\in T}P_{t}$ for the special fibre of $\sX\to R$ and denote by $\tilde{P}_{t}\subset\tilde{\sX_{0}}$ the irreducible component that maps isomorphically onto $P_{t}$ via the composition $f\colon\tilde{\sX}\to\sX_{R}\to\sX$ for every $t\in T$. By Lemma \ref{lem:2.2}, we have the commutative diagram:\begin{center}\begin{tikzcd}
{\br(\sX_{\eta})[\ell^{r}]} \arrow[r, "\{\partial_{t}\}"] \arrow[d, "f^{*}"] & {\oplus_{t\in T}H^{1}(k(P_{t}),\bbZ/\ell^{r})} \arrow[d, "\{\ell^{r}f^{*}\}=0"] \\
{\br(\tilde{\sX}_{\eta})[\ell^{r}]} \arrow[r, "\{\partial_{t}\}"]            & {\oplus_{t\in T}H^{1}(k(\tilde{P}_{t}),\bbZ/\ell^{r})},           
\end{tikzcd}\end{center}from which one obtains $\partial_{t}f^{*}\alpha=0$ for all $t\in T$. It remains to check that $\tilde{\alpha}\coloneqq f^{*}\alpha\in\br(\tilde{\sX}_{\eta})[\ell^{r}]$ has trivial residue along each irreducible component that has been introduced by resolving the singularities that appeared due to the base change. Over a new component the restriction of $\tilde{q}$ is trivial. It follows that for such a component $V\subset\tilde{\sX}_{0}$, if we write $\tilde{q}^{-1}(V)\cong\bigsqcup_{g\in G} \{g\}\times V$ and let $\tilde{q_{g}}\colon\{g\}\times V\to V $ be the restriction of $\tilde{q}$ to $\{g\}\times V$, then $0=(\tilde{q_{g}})_{*}(\partial_{g}\tilde{q_{\eta}}^{*}\tilde{\alpha})=(\tilde{q_{g}})_{*}\tilde{q_{g}}^{*}(\partial_{V}\tilde{\alpha})=\partial_{V}\tilde{\alpha}\in H^{1}(k(V),\bbZ/\ell^{r})$. Here, the first equality holds, because $\tilde{q_{\eta}}^{*}\tilde{\alpha}=0$, whereas for the last one, we use that $(\tilde{q_{g}})_{*}\tilde{q_{g}}^{*}=\id$ on $V$. Finally, the lift $\tilde{\alpha}\in\br(\tilde{\sX})\{\ell\}$ satisfies $\tilde{q}^{*}\tilde{\alpha}=0\in\br(\tilde{\sY})$, as the restriction of $\tilde{q}^{*}\tilde{\alpha}$ to $\tilde{\sY_{\eta}}$ is zero and the natural pull-back map $\br(\tilde{\sY})\to\br(\sY_{\eta})$ is injective. The proof now is complete.\end{proof}\section{Brauer Group of the Total space and a specialization argument}\label{sec:4}Let $p\colon\sX\to\Spec R$ be a strictly semi-stable scheme over a henselian discrete valuation ring $R$ with seperably closed residue field $k$. The main aim of this section is the study of the restriction map $\br(\sX)\to\br(\sX_{0})_{\tor}$.\begin{lemma}\label{Lem:4.1} Let $p\colon\sX\to\Spec R$ be a strictly semi-stable scheme over a henselian discrete valuation ring $R$ with seperably closed residue field $k$. Assume that $H^{1}(X,\sO_{X})=0$ and $b_{2}=\rho$, where $b_{2}$ is the second $\ell$-adic betti number of $X\coloneqq\sX_{\bar{\eta}}$ for some prime $\ell$ invertible in $R$ and $\rho$ is the rank of the N\'eron severi group of $X$. Then for any prime $\ell$ invertible in $R$, restriction yields an isomorphism $\br(\sX)\{\ell\}\to\br(\sX_{0})\{\ell\}.$\end{lemma}\begin{proof} We follow the proof in \cite[Proposition 7.2]{Stef}. In what follows, $H^{i}_{cont}$ stands for Jannsen's continuous \'etale cohomology \cite{Jan}. The proper base change theorem (see \cite[Corollary VI.2.7]{mil}) gives rise to the natural isomorphisms\[H^{i}(\sX,\mu_{\ell^{r}})\cong H^{i}(\sX_{0},\mu_{\ell^{r}}) \tag{4.1}\label{4.1}\]\[H_{cont}^{i}(\sX,\bbZ_{\ell}(1))\cong H^{i}(\sX_{0},\bbZ_{\ell}(1)).\tag{4.2}\label{4.2}\] Note that \eqref{4.2} follows from \eqref{4.1}. Indeed, \eqref{4.1} implies that the groups $H^{i}(\sX,\mu_{\ell^{r}})$ are finite (see \cite[Corollary VI.2.8]{mil}) and thus, in this case Jannsen's continuous \'etale cohomology coincides with the usual $\ell$-adic cohomology (see \cite[(0.2)]{Jan}).\par We prove that the $\ell$-adic cycle class map $c_{\ell,1}\colon\Pic(\sX)\otimes\bbZ_{\ell}\to H^{2}(\sX,\bbZ_{\ell}(1))$ is surjective. To this end, we compute $H^{2}_{cont}(\sX_{\eta},\bbZ_{\ell}(1))$, using the Hochschild-Serre spectral sequence (see \cite[Corollary 3.4]{Jan})\[E^{p,q}_{2}\coloneqq H^{p}_{cont}(G,H^{q}(X,\bbZ_{\ell}(1)))\implies H^{p+q}_{cont}(\sX_{\eta},\bbZ_{\ell}(1)), \tag{4.3}\label{4.3}\] where $G\coloneqq\Gal(K^{sep}/K)$ is the absolute Galois group of the fraction field $K$ of $R$ and $X\coloneqq\sX_{\bar{\eta}}$. We observe that \[H^{1}(X,\bbZ_{\ell}(1))\coloneqq\lim_{\underset{r}\longleftarrow} H^{1}(X,\mu_{\ell^{r}})\cong\lim_{\underset{r}\longleftarrow}\Pic(X)[\ell^{r}]\eqqcolon T_{\ell}\Pic(X),\tag{4.4}\label{4.4}\]where the isomorphism is canonical and is induced from the Kummer sequence (see \cite[Corollary 4.18]{mil}). Since $H^{1}(X,\sO_{X})=0$, the identity component of the Picard scheme $\Pic^{0}_{X/\bar{K}}$ is just a point (see \cite[Corollary 5.13]{kl}) and so $\Pic(X)=\NS(X)$. The N\'eron Severi group of a proper variety over an algebraically closed field is finitely generated (see \cite[Theorem V.3.25]{mil}). Hence, the torsion subgroup of $\Pic(X)$ is finite and thus $T_{\ell}\Pic(X)=0$. Therefore, \eqref{4.4} gives $E^{p,1}_{2}=0$ for every $p\geq 0$.\par Next, we calculate $E^{2,0}_{2}=H^{2}_{cont}(G,\bbZ_{\ell}(1))$. Consider the Gysin sequence
\[H^{1}(R,\mu_{\ell^{r}})\to H^{1}(G,\mu_{\ell^{r}})\overset{\partial}\to H^{0}(k,\bbZ/\ell^{r})\to H^{2}(R,\mu_{\ell^{r}}).\tag{4.5}\label{4.5}\] Proper base change (see \cite[Corollary VI.2.7]{mil}) together with $k^{sep}=k$, gives $H^{1}(R,\mu_{\ell^{r}})\cong H^{1}(k,\mu_{\ell^{r}})\cong k^{*}/(k^{*})^{\ell^{r}}=0$ and $H^{2}(R,\mu_{\ell^{r}})\cong H^{2}(k,\mu_{\ell^{r}})=\br(k)[\ell^{r}]=0$. Hence, \eqref{4.5} yields an isomorphism $\partial\colon H^{1}(G,\mu_{\ell^{r}})\overset{\cong}\to \bbZ/\ell^{r}$. The latter shows that the groups $H^{1}(G,\mu_{\ell^{r}})$ are finite for all $r$ and so we obtain $H^{2}_{cont}(G,\bbZ_{\ell}(1))=\lim H^{2}(G,\mu_{\ell^{r}})$ (see \cite[(2.1)]{Jan}). By \cite[Proposition 1.4.5]{cts}, we have that $cd_{\ell}(K)\leq 1$ and thus, $H^{2}_{cont}(G,\bbZ_{\ell}(1))=0$.\par From the above calculations, we find that $E^{p,1}_{\infty}=E^{p,1}_{2}=0$ and $E^{2,0}_{\infty}=E^{2,0}_{2}=0$. In addition, we note that $E^{0,2}_{\infty}=E^{0,2}_{2}$. To see this, we need to show that the term $E^{3,0}_{\infty}=E^{3,0}_{2}$ also vanishes. But this follows again from \cite[(2.1)]{Jan}, using that $cd_{\ell}(K)\leq 1$. Consequently, \eqref{4.3} yields an isomorphism \[H^{2}_{cont}(\sX_{\eta},\bbZ_{\ell}(1))\cong H^{2}(X,\bbZ_{\ell}(1))^{G}.\tag{4.6}\label{4.6}\]\par We consider the following short exact sequence induced from the Kummer exact sequence\[0\to\Pic(X)\otimes\bbZ/\ell^{r}\to H^{2}(X,\mu_{\ell^{r}})\to\br(X)[\ell^{r}]\to 0\tag{4.7}\label{4.7}.\] Taking the inverse limit of \eqref{4.7} over $r$, we obtain the short exact sequence \[0\to\Pic(X)\otimes\bbZ_{\ell}\overset{c_{\ell,1}}\to H^{2}(X,\bbZ_{\ell}(1))\to T_{\ell}\br(X)\to0\tag{4.8}\label{4.8}.\] The last term of \eqref{4.8} is a free $\bbZ_{\ell}$-module of finite rank. Tensoring \eqref{4.8} with $\bbQ_{\ell},$ we see that the rank of $T_{\ell}\br(X)$ is $b_{2}-\rho$ and so vanishes. As clearly $\sX^{\sm}_{0}(k)\neq\emptyset$, the henselian property of $R$ implies that the $R$-scheme $\sX\to\Spec R$ admits a section (see \cite[Theorem I.4.2]{mil}). From this follows that the relative Picard functor $\Pic_{\sX_{\eta}/K}$ is representable by a $K$-scheme (see \cite[Theorem 2.5]{kl} and \cite[Theorem 4.8]{kl}) and hence, $\Pic(X)^{G}={\Pic_{\sX_{\eta}/K}(K^{sep})}^{G}=\Pic_{\sX_{\eta}/K}(K)=\Pic(\sX_{\eta})$. Passing to $G$-invariants the $\ell$-adic cycle class map induces by \eqref{4.6} an isomorphism $c_{\ell,1}\colon\Pic(\sX_{\eta})\otimes\bbZ_{\ell}\to H_{cont}^{2}(\sX_{\eta},\bbZ_{\ell}(1))$.\par The Gysin sequence for Jannsen's continuous \'etale cohomology \cite[\S 6]{stefanos}, yields an exact sequence\[\oplus^{n}_{i=1}\bbZ_{\ell}[\sX_{0,i}]\to H^{2}(\sX,\bbZ_{\ell}(1))\to H_{cont}^{2}(\sX_{\eta},\bbZ_{\ell}(1)).\tag{4.9}\label{4.9}\] Since we know now that $H_{cont}^{2}(\sX_{\eta},\bbZ_{\ell}(1))$ is algebraic, taking closures of divisors on $\sX_{\eta}$ in $\sX$ shows via \eqref{4.9} that $H^{2}(\sX,\bbZ_{\ell}(1))$ is algebraic, too. By \eqref{4.2} and the functoriality with respect to pullbucks of the $\ell$-adic cycle class map, we obtain similarly that $c_{\ell,1}\colon\Pic(\sX_{0})\otimes\bbZ_{\ell}\to H^{2}(\sX_{0},\bbZ_{\ell}(1))$ is surjective. In particular, \eqref{2.6} implies \[\coker(H^{2}(\sX,\bbZ_{\ell}(1))\to H^{2}(\sX,\mu_{\ell^{r}}))\cong \br(\sX)[\ell^{r}]\]\[\coker(H^{2}(\sX_{0},\bbZ_{\ell}(1))\to H^{2}(\sX_{0},\mu_{\ell^{r}}))\cong \br(\sX_{0})[\ell^{r}].\] The natural isomorphism $\br(\sX)[\ell^{r}]\cong\br(\sX_{0})[\ell^{r}]$ then follows from \eqref{4.1} and \eqref{4.2}.\end{proof}\begin{remark}\label{rem:4.2} Regarding the characteristic zero case, the conditions in the above Lemma are satisfied, when $H^{i}(X,\sO_{X})=0$ for $i=1,2$.\end{remark} In general, $\rho\leq b_{2}$ and one can study the restriction map $\br(\sX)\to\br(\sX_{0})_{\tor}$ via the following diagram with exact rows:
\begin{center}\begin{tikzcd}
(4.10) & 0 \arrow[r] & \Pic(\sX)\otimes \bbZ/\ell^{r} \arrow[r, "c_{1}"] \arrow[d,  hook] & {H^{2}(\sX,\mu_{\ell^{r}})} \arrow[r] \arrow[d, "\cong"] & {\br(\sX)[\ell^{r}]} \arrow[r] \arrow[d,  two heads] & 0 \\
       & 0 \arrow[r] & \Pic(\sX_{0})\otimes \bbZ/\ell^{r} \arrow[r, "c_{1}"]                      & {H^{2}(\sX_{0},\mu_{\ell^{r}})} \arrow[r]                & {\br(\sX_{0})[\ell^{r}]} \arrow[r]                           & 0,
\end{tikzcd}
\end{center}
where $\ell$ is any prime invertible in $R$ and the middle restriction map is isomorphism by the proper base change theorem (see \cite[Corollary VI.2.7]{mil}). Applying the snake Lemma to the above diagram, we deduce the isomorphism\[\ker(\br(\sX)\to\br(\sX_{0}))[\ell^{r}]\cong \coker(\Pic(\sX)\to\Pic(\sX_{0}))\otimes\bbZ/\ell^{r}.\tag{4.11}\label{4.13}\]Passing to the colimit over $r$, we find that the map $\br(\sX)\{\ell\}\to\br(\sX_{0})\{\ell\}$ is isomorphism if and only if \[\frac{\Pic(\sX_{0})}{\Pic(\sX)}\otimes\bbQ_{\ell}/\bbZ_{\ell}=0.\tag{4.12}\label{4.14}\]Elementary deformation theory thus gives the following result.\begin{lemma}\label{Lem:4.3}Let $p\colon\sX\to\Spec R$ be a strictly semi-stable scheme over a complete discrete valuation ring $R$. If $H^{2}(\sX_{0},\sO_{\sX_{0}})=0$, then for any prime $\ell$ invertible in $R$, the map $\br(\sX)\{\ell\}\to\br(\sX_{0})\{\ell\}$ is an isomorphism.\end{lemma}\begin{proof}The obstruction to extending a line bundle lies in $H^{2}(\sX_{0},\sO_{\sX_{0}})$. Since this is zero, we may lift any line bundle $\sL_{0}$ on $\sX_{0}$ to the total space $\sX$. Thus, the result follows immediately from \eqref{4.14}.\end{proof}\begin{lemma}\label{lem:4.4} Let $k$ be a field and set $p=\Char(k)$ if $\Char(k)>0$ and $p=1,$ otherwise. Let $\sZ\to B$ be a smooth proper family of $k$-varieties with geometrically irreducible fibres. Then for $t\in B$ and a prime $\ell\neq p,$ we have a specialization map on Brauer groups $sp_{\bar{\eta},\bar{t}}\colon \br(\sZ_{\bar{\eta}})\{\ell\}\to \br(\sZ_{\bar{t}})\{\ell\},$ which is functorial with respect to pullbacks. In addition, this map is always surjective and is an isomorphism if and only if $\rho(\sZ_{\bar{\eta}})=\rho(\sZ_{\bar{t}})$.\end{lemma}\begin{proof} Fix a point $t\in B$ and a prime $\ell\neq p$. Recall that we have specialization maps $sp_{\bar{\eta},\bar{t}}\colon \NS(\sZ_{\bar{\eta}})\to\NS(\sZ_{\bar{t}})$ (see \cite[\S 20.3]{fw}) and $sp_{\bar{\eta},\bar{t}}\colon H^{2}(\sZ_{\bar{\eta}},\mu_{\ell^{r}})\to H^{2}(\sZ_{\bar{t}},\mu_{\ell^{r}}),$ that are both functorial with respect to pullbacks. The second one is an isomorphism and is induced by the smooth proper base change theorem (see \cite[Corollary VI.4.2]{mil}). As specialization commutes with the cycle class maps (see \cite[App. X]{gr}), \eqref{2.6} yields the homomorphism $sp_{\bar{\eta},\bar{t}}\colon \br(\sZ_{\bar{\eta}})[\ell^{r}]\to \br(\sZ_{\bar{t}})[\ell^{r}]$. The functoriality of this homomorphism with respect to pullbacks is clear.\par We have the following diagram with exact rows\[\begin{tikzcd}
0 \arrow[r] & \NS(\sZ_{\bar{\eta}})\otimes\bbZ/\ell^{r} \arrow[r] \arrow[d, "{sp_{\bar{\eta},\bar{t}}}"] & {H^{2}(\sZ_{\bar{\eta}},\mu_{\ell^{r}})} \arrow[r] \arrow[d, "{sp_{\bar{\eta},\bar{t}}}"] & {\br(\sZ_{\bar{\eta}})[\ell^{r}]} \arrow[r] \arrow[d, dashed, "{sp_{\bar{\eta},\bar{t}}}"] & 0 \\
0 \arrow[r] & \NS(\sZ_{\bar{t}})\otimes\bbZ/\ell^{r} \arrow[r]                                           & {H^{2}(\sZ_{\bar{t}},\mu_{\ell^{r}})} \arrow[r]                                           & {\br(\sZ_{\bar{t}})[\ell^{r}]} \arrow[r]                                           & 0.\tag{4.13}\label{4.15}\end{tikzcd}\] Since the middle vertical map is an isomorphism, the first vertical map is always injective and the last one is surjective.\par Next, we assume $\rho(\sZ_{\bar{\eta}})=\rho(\sZ_{\bar{t}})$ and want to show that $sp_{\bar{\eta},\bar{t}}\colon \br(\sZ_{\bar{\eta}})\{\ell\}\to \br(\sZ_{\bar{t}})\{\ell\}$ is an isomorphism. All the terms in \eqref{4.15} are finite, and so taking inverse limits over $r$ is exact. Thus, the last vertical map in \eqref{4.15} induces a surjection between the associated $\ell$-adic Tate modules. As both $T_{\ell}\br(\sZ_{\bar{\eta}})$ and $T_{\ell}\br(\sZ_{\bar{t}})$ are free $\bbZ_{\ell}$-modules of the same rank $b_{2}-\rho$, the map $T_{\ell}\br(\sZ_{\bar{\eta}})\to T_{\ell}\br(\sZ_{\bar{t}})$ is in fact an isomorphism. Applying the snake Lemma to the inverse limit of \eqref{4.15}, the latter also yields that the specialization map $sp_{\bar{\eta},\bar{t}}\colon \NS(\sZ_{\bar{\eta}})\otimes\bbZ_{\ell}\to\NS(\sZ_{\bar{t}})\otimes\bbZ_{\ell}$ is an isomorphism for every $\ell\neq p$. The faithfully flat property of $\bigsqcup_{\ell\neq p}\Spec\bbZ_{\ell}\to\Spec\bbZ[1/p],$ implies $sp_{\bar{\eta},\bar{t}}\colon \NS(\sZ_{\bar{\eta}})\otimes\bbZ[1/p]\to\NS(\sZ_{\bar{t}})\otimes\bbZ[1/p]$ is an isomorphism. As a consequence, we also obtain $sp_{\bar{\eta},\bar{t}}\colon\br(\sZ_{\bar{\eta}})\{\ell\}\to\br(\sZ_{\bar{t}})\{\ell\}$ is an isomorphism for all $\ell\neq p$, since this holds for the first two vertical maps in \eqref{4.15}.\par The converse is true as well. Namely, if $sp_{\bar{\eta},\bar{t}}\colon\br(\sZ_{\bar{\eta}})\{\ell\}\to\br(\sZ_{\bar{t}})\{\ell\}$ is an isomorphism, then applying the snake Lemma to \eqref{4.15}, one gets $\rho(\sZ_{\bar{\eta}})=\rho(\sZ_{\bar{t}})$.\end{proof}
\begin{proposition}\label{prp:4.4} Let $k$ be a field and set $p=\Char(k)$ if $\Char(k)>0$ and $p=1,$ otherwise. Let $f\colon\sX\to B$ and $g\colon\sY\to B$ be smooth proper morphisms of $k$-varieties with geometrically irreducible fibres. Let $q\colon\sY\to\sX$ be a morphism of $B$-schemes. Assume that $\rho(\sX_{\bar{\eta}})=b_{2},$ where $b_{2}$ is the second $\ell$-adic betti number of $\sX_{\bar{\eta}}$ for some prime $\ell\neq p$. For $t\in B$, consider the pull-back map $q_{\bar{t}}^{*}\colon\br(\sX_{\bar{t}})\to\br(\sY_{\bar{t}})$. Then for every prime $\ell\neq p$ and $t\in B$, such that $\rho(\sY_{\bar{t}})=\rho(\sY_{\bar{\eta}}),$ we have $\ker(q_{\bar{t}}^{*})\{\ell\}\cong \ker (q_{\bar{\eta}}^{*})\{\ell\}$.\end{proposition}\begin{proof}We have $\rho(\sX_{\bar{\eta}})\leq \rho(\sX_{\bar{t}})\leq b_{2}$ for all $t\in B$. The inequalities follow from the injectivity of $sp_{\bar{\eta},\bar{\tau}}\colon \NS(\sX_{\bar{\eta}})\otimes\bbZ[1/p]\to\NS(\sX_{\bar{t}})\otimes\bbZ[1/p]$ and of the $\ell$-adic cycle class map $c_{\ell,1}$ for a prime $\ell\neq p$. The assumption $\rho(\sX_{\bar{\eta}})=b_{2}$ implies that the Picard number in the family $\sX\to B$ is constant and equals $b_{2}$. By Lemma \ref{lem:4.4}, we deduce that for $t\in B,$ with $\rho(\sY_{\bar{t}})=\rho(\sY_{\bar{\eta}})$, both $sp_{\bar{\eta},\bar{t}}\colon\br(\sX_{\bar{\eta}})\{\ell\}\to\br(\sX_{\bar{t}})\{\ell\}$ and $sp_{\bar{\eta},\bar{t}}\colon\br(\sY_{\bar{\eta}})\{\ell\}\to\br(\sY_{\bar{t}})\{\ell\}$ are isomorphisms. The result follows as $sp_{\bar{\eta},\bar{t}}\circ q^{*}_{\bar{\eta}}=q_{\bar{t}}^{*}\circ sp_{\bar{\eta},\bar{t}}$.\end{proof}\section{Main Application: Godeaux Surface}\label{sec:5} Let $k$ be an algebraically closed field of characteristic $0$. Any quintic surface $Y,$ which is invariant and fixed point free under the $\bbZ/5$-action \eqref{1.1} can be realised as a hypersurface in $\bbP_{k}^{3}$ defined by an equation of the form\begin{center}$G_{\alpha_{1},\ldots,\alpha_{8}}(T_{0},\ldots,T_{3})= T^{5}_{0}+T^{5}_{1}+T^{5}_{2}+T^{5}_{3}+\alpha_{1}T^{3}_{0}T_{2}T_{3}+\alpha_{2}T_{0}T^{3}_{1}T_{2}+\alpha_{3}T_{1}T^{3}_{2}T_{3}+\alpha_{4}T_{0}T_{1}T^{3}_{3}+\alpha_{5}T_{0}T^{2}_{2}T^{2}_{3}+\alpha_{6}T^{2}_{0}T_{1}T^{2}_{2}+\alpha_{7}T^{2}_{1}T_{2}T_{3}^{2}+\alpha_{8}T^{2}_{0}T^{2}_{1}T_{3},$\end{center}where $\alpha\coloneqq(\alpha_{1},\alpha_{2},\ldots,\alpha_{8})\in\bbA_{k}^{8}(k)$. Given a point $\alpha\in\bbA_{k}^{8}(k)$, we set $Y_{\alpha}\coloneqq\bbV(G_{\alpha})\subset\bbP^{3}_{k}$.\begin{proposition}[\cite{persson}]\label{prp:5.1}Let $k$ be an algebraically closed field of characteristic $0$. There exists a triple-point free semi-stable degeneration $\sX\to\Spec k[[t]]$ of a Godeaux surface, such that the dual graph $\Gamma$ of the special fibre $\sX_{0}$ is a chain. If we number the components $\sX_{0,i}$ of $\sX_{0}$, so that $\sX_{0,i}\cap\sX_{0,i+1}\neq\emptyset$ for every $i$, then all $\sX_{0,i}$ are elliptic ruled surfaces except the last one $\sX_{0,n}$, which is a cubic surface. The components $\sX_{0,i}$ and $\sX_{0,i+1}$ are glued along smooth elliptic curves $C_{i,i+1}$. In addition, if $C_{\sX_{0,i}}\coloneqq\sX_{0,i}\cap\sX^{\sing}_{0},$ then $C_{\sX_{0,1}}$ sits as a $5$-fold cover onto the base of the ruled surface $\sX_{0,1}$, $C_{\sX_{0,i}}$ consists of two disjoint sections of the ruled surface $\sX_{0,i}$ for $i=2,\ldots,n-1$ and $C_{\sX_{0,n}}+K_{\sX_{0,n}}=0\in\Pic(\sX_{0,n})$.\end{proposition}\begin{proof}We briefly sketch Perssons' construction (see \cite[Appendix 3]{persson}). Pick $p\in\bbP^{3}_{k}(k)$ not in any coordinate hyperplane. Then there exists $\alpha_{0}\in\bbA^{8}_{k}(k),$ such that $Y_{0}\coloneqq Y_{\alpha_{0}}$ has $5$ triple points at the $\bbZ/5$-orbit of $p$ and no other singularities (see \cite[Appendix 3, Lemma 1]{persson}). The resolution $\nu\colon\tilde{Y}_{0}\to Y_{0}$ obtained by blowing up the $5$ triple points is an elliptic ruled surface and the curves $E_{i}\coloneqq\nu^{-1}(p_{i})$, where $p_{i}$ is a singular point of $Y_{0},$ are smooth elliptic and constitute sections in $\tilde{Y}_{0}$ (see \cite[Appendix 3, Lemma 2]{persson}). Since we blow up with respect to a $\bbZ/5$-invariant center, the $\bbZ/5$-action on $Y_{0}$ lifts naturally to $\tilde{Y}_{0}$. This action has no fixed points on $\tilde{Y}_{0}$ and permutes the curves $E_{i}$. Hence, passing to the quotient by the $\bbZ/5$-action on $\tilde{Y}_{0},$ we obtain an elliptic ruled surface $\tilde{X}_{0}$ and the curves $E_{i}$ become a $5$-fold cover onto the base of $\tilde{X}_{0}$ (see \cite[Appendix 3, Lemma 2]{persson}).\par Choose $\alpha_{1}\in\bbA^{8}_{k}(k),$ such that $Y_{1}\coloneqq Y_{\alpha_{1}}$ is smooth and $p\notin Y_{1}$. Consider the one-parameter family of $\bbZ/5$-invariant quintics $\sY\coloneqq\bbV(G_{\alpha_{0}}+t(G_{\alpha_{1}}-G_{\alpha_{0}}))\subset\bbP^{3}_{k}\times_{k}\bbA^{1}_{k}$. The condition $p\notin Y_{1}$ implies that $\sY$ is smooth along $Y_{0}$. We perform a base change with respect to the local ring at $0\in\bbA^{1}_{k}(k)$ and note that the total space of $\sY\to B$, $B\coloneqq\Spec\sO_{\bbA^{1}_{k},0},$ remains regular. Let $\tilde{\sY}\to\sY$ be the blow up at the $5$ triple points of $Y_{0}$. Then the special fibre of the degeneration $\tilde{\sY}\to B$ is given by $\tilde{Y}_{0}$ and along each elliptic curve $E_{i}$, there is glued a $P_{i}\cong\bbP^{2}$ with multiplicity $3$. Consequently, $\sO_{\tilde{\sY}}(\tilde{Y}_{0})\cong \sO_{\tilde{\sY}}(-\sum_{i}P_{i})^{\otimes 3}$. We take the triple cyclic covering $\sY^{'}\to\tilde{\sY}$ branched along $\tilde{Y}_{0}$. Then $\sY'$ is regular and the composite $\sY^{'}\to B$ factors through $ B\to B, t\mapsto t^{3},$ giving rise to a degeneration $\sY'\to B$ of a $\bbZ/5$-invariant quintic. The special fibre of the resulting degeneration consists of $\tilde{Y}_{0}$ and along each $E_{i}$, there is glued a cubic surface $R_{i}$ with multiplicity $1$. Finally, forming the $\bbZ/5$-quotient, we deduce a semi-stable degeneration $\sX\to B$ of a Godeaux surface. The special fibre of this family consists of two components, the elliptic ruled surface $\tilde{X}_{0}$ and a cubic surface $R$. They intersect along a smooth elliptic curve $C,$ which sits as a $5$-fold section in $\tilde{X}_{0}$ and satisfies $K_{\tilde{X}_{0}}+C=0\in\Pic(\tilde{X}_{0})$.\end{proof} 
\begin{proposition}\label{prp:5.2} Let $k$ be an algebraically closed field of characteristic $0$. Let $\sX\to\Spec k[[t]]$ be a semi-stable degeneration of a Godeaux surface with special fibre as in Proposition \ref{prp:5.1}. Then $\sX\to\Spec k[[t]]$ admits a natural $5:1$ \'etale cover $\sY\to\Spec k[[t]]$, which is a triple-point free semi-stable degeneration of a smooth quintic surface. The dual graph $\Gamma$ of the special fibre $\sY_{0}$ is a flower pot. In particular, $\sY_{0}$ is given by an elliptic ruled surface $Y_{0}$ with five disjoint sections $E_{i},$ such that along each $E_{i},$ there is glued a chain of elliptic ruled surfaces $Y_{i,j}$ with a cubic surface $Y_{i,n}$ as an end-component. These chains are isomorphic to each other and they are permuted by the $\bbZ/5$-action of the covering, whereas $Y_{0}$ is invariant under this action.\end{proposition}
\begin{proof}The Mayer-Vietoris sequence \eqref{2.4} gives rise to the following exact sequence\[0\to\Pic(\sX_{0})\to\bigoplus_{i=1}^{n}\Pic(\sX_{0,i})\overset{r}\to\bigoplus_{i=1}^{n-1}\Pic(C_{i,i+1})\tag{5.1}\label{5.1}.\]\par We aim to compute the torsion group $\Pic(\sX_{0})_{\tor}$. Since $\sX_{0,n}$ is a rational surface, we have $\Pic(\sX_{0,n})_{\tor}=0$. Thus, the restriction $r_{\tor}$ of $r$ to $\bigoplus_{i=1}^{n}\Pic(\sX_{0,i})_{\tor},$ takes $(\sL_{i})_{i}$ to $(\sL_{1}|_{C_{1,2}}-\sL_{2}|_{C_{1,2}},\ldots,\sL_{n-1}|_{C_{n-1,n}})$. Recall that for $i=2,\ldots,n-1,$ the components $\sX_{0,i}$ are ruled and both curves $C_{i-1,i}$ and $C_{i,i+1}$ constitute sections in $\sX_{0,i}$. Hence, restriction of line bundles yields isomorphisms $\Pic(\sX_{0,i})_{\tor}\cong\Pic(C_{i-1,i})_{\tor}$ (resp. $ \Pic(C_{i,i+1})_{\tor}$). It follows that the kernel of $r_{\tor}$ coincides with the one of $\Pic(\sX_{0,1})_{\tor}\to\Pic(C_{1,2})_{\tor},\sL\mapsto\sL|_{C_{1,2}}$. The last map can be identified with ${(\psi|_{C_{1,2}})}^{*}\colon\Pic(E)_{\tor}\cong\Pic(\sX_{0,1})_{\tor}\to\Pic(C_{1,2})_{\tor},$ where $\psi\colon\sX_{0,1}\to E$ is the fibering of the ruled surface $\sX_{0,1}$. Since $C_{1,2}$ is a $5$-fold cover onto the base curve $E$, the kernel of the pullback map ${(\psi|_{C_{1,2}})}^{*}$ is the Cartier dual of $\bbZ/5$, which is $\mu_{5}\cong\bbZ/5$. Consequently, $\Pic(\sX_{0})_{\tor}\cong\bbZ/5$.\par To conclude, we need to show that a non-trivial $5$-torsion line bundle on $\sX_{0}$ lifts to a $5$-torsion line bundle on the total space $\sX$. But torsion line bundles always extend to torsion ones (see \cite[Proposition 22.5]{hr}). Lastly, the desired form of the special fibre for the associated cover $\sY$ is justified by the fact that the restriction of $0\neq\sL\in\Pic(\sX_{0})_{\tor}\cong\bbZ/5$ to $\cup_{i>1}\sX_{0,i}$ is trivial, whereas $\sL|_{\sX_{0,1}}\neq 0$.\end{proof}\begin{lemma}\label{lem:5.2} Let $k$ be an algebraically closed field of characteristic $0$. Let $\sX\to\Spec k[[t]]$ be a semi-stable degeneration of a Godeaux surface with special fibre $\sX_{0},$ as in Proposition \ref{prp:5.1}. Then $\br(\sX)=0$.\end{lemma}\begin{proof} A Godeaux surface $X$ over $\overline{k((t))}$ satisfies $H^{i}(X,\sO_{X})=0$ for $i=1,2$. Hence, we may apply Lemma \ref{4.1} and obtain the isomorphism $\br(\sX)\to\br(\sX_{0})_{\tor}$.\par We compute $\br(\sX_{0})$ with the help of the Mayer-Vietoris exact sequence \eqref{2.4}. As the components of the special fibre $\sX_{0}$ are all ruled or rational and the ground field is algebraically closed, we find $\br(\sX_{0,i})=0$ for all $i$. Thus, \eqref{2.4} gives rise to the exact sequence \[\bigoplus^{n}_{i=1}\Pic(\sX_{0,i})\overset{r}\to\bigoplus^{n-1}_{i=1}\Pic(C_{i,i+1})\to \br(\sX_{0})\to 0,\tag{5.2}\label{5.2}\] and we need to show that $r$ is surjective. Consider the fibrations $p_{i}\colon\sX_{0,i}\to E_{i}$ and let $s_{i,j}\colon E_{i}\to C_{j,j+1}\subset\sX_{0,i}$ be the corresponding sections, where $j=i-1,i$ and $i=2,\ldots,n-1$. For a given point $x\in E_{i}$, we set $F^{i}_{x}\coloneqq p^{-1}_{i}(x)$ for the fibre of $p_{i}$ over $x$ and note that, \[r(0,\ldots,[F^{i}_{x}],\ldots,0)=(0,\ldots,-[s_{i,i-1}(x)],[s_{i,i}(x)],\ldots,0).\tag{5.3}\label{5.3}\] As the double curves $C_{i,i+1}$ are isomorphic to each other, \eqref{5.3} implies that it is enough to show $([x],0,\ldots,0)\in\im(r)$ for all points $x\in C_{1,2}$. In fact, it suffices to find a single point $x\in C_{1,2}$ for which the latter holds. Indeed, this follows from $\Pic(C_{1,2})=\bbZ[x]\oplus\Pic^{0}(C_{1,2})$ and that the composite $\Pic^{0}(E_{1})\subset\Pic^{0}(\sX_{0,1})\to\Pic^{0}(C_{1,2})$ is surjective.\par Since $\sX_{0,n}$ is a cubic surface, we may pick a $(-1)$-rational curve $D\subset\sX_{0,n}$. Recall that $K_{\sX_{0,n}}=- C_{n-1,n}$. Thus, the adjunction formula $(K_{\sX_{0,n}}+D).D=-2$ yields $C_{n-1,n}.D=1$ and so $r(0,\ldots,0,-[D])=(0,\ldots,[x])$ for some point $x\in C_{n-1,n}$. Finally, by \eqref{5.3} we get a point $x\in C_{1,2}$, such that $([x],0,\ldots,0)\in\im(r)$, finishing the proof.\end{proof} We are in the position to prove the main result of this paper.\begin{proof}[Proof of Theorem \ref{thm:1.3}.] Pick a generic pencil $\sY\to\bbP^{1}_{k}\subset|\sO_{\bbP^{3}_{k}}(5)|^{\bbZ/5}$ of $\bbZ/5$-invariant quintics, such that some fibre $\sY_{t_{0}},$ $t_{0}\in\bbP^{1}_{k}(k)$ is not passing through any of the $4$ fixed points of the action \eqref{1.1} and has $5$ triple points as its only singularities. Since the pencil is generic, we may assume that the $5$ triple points are not base points and so the total space $\sY$ is smooth along $\sY_{t_{0}}$. We perform a base change with respect to the completion of the local ring of $\bbP^{1}_{k}$ at $t_{0}$. Applying semi-stable reduction to the family $\sY$ and its $\bbZ/5$-quotient $\sX$, we arrive at the $5:1$ \'etale cover $q\colon\tilde{\sY}\to\tilde{\sX},$ described in Proposition \ref{prp:5.2} (see proof of Proposition \ref{prp:5.1}). We want to show that the pullback map $q^{*}_{\bar{\eta}}$ is injective. If $q^{*}_{\bar{\eta}}=0$, then Theorem \ref{thm:1.4} implies that up to some finite base change of discrete valuation rings followed by a resolution of the total space (see \cite{Har}), the restriction map $\br(\tilde{\sX})\to\br(\sX_{\bar{\eta}})$ is surjective. The latter contradicts Lemma \ref{lem:5.2}.\end{proof}\begin{remark}\label{rem:5.4}In general the universal cover $Y$ of a Godeaux surface $X$ is not a quintic, but rather its canonical model $\nu\colon Y\to Y_{can}$ that is obtained by contracting the $(-2)$-curves (see \cite[Theorem 2]{lw}). In particular, $Y_{can}\subset\bbP^{3}_{\bbC}$ is invariant, fixed point free under the $\bbZ/5$-action \eqref{1.1} and $X_{can}=Y_{can}/(\bbZ/5)$ (see \cite[$\S 2$]{lw}).\par We note that if $\nu\colon Y\to Y_{can}$ is not an isomorphism, then $\rho(Y)\geq 13$. Especially, $\rho(Y)=9$ implies that $Y$ is a quintic surface.\par Indeed, let $n$ denote the number of dinstinct $\bbZ/5$-orbits of rational double points in $Y_{can}$. Pick a representative $p_{i}\in Y_{can}$ for each $\bbZ/5$-orbit and let $n_{i}$ be the number of components of the fibre $E_{i}\coloneqq \nu^{-1}(p_{i})$, where $i=1,\ldots,n$. Then we have the relations \[\rho(Y)=\rho(Y_{can})+\sum^{n}_{i=1}5n_{i}\ \text{and}\ 9=\rho(X)=\rho(X_{can})+\sum^{n}_{i=1}n_{i}.\]Since $\rho(Y_{can})\geq\rho(X_{can}),$ we find $\rho(Y)\geq 9 +\sum^{n}_{i=1}4n_{i}\geq 13$, as claimed.\end{remark}\begin{proof}[Proof of Theorem \ref{thm:1.1}.]The generic Picard number in the family of smooth, $\bbZ/5$-invariant and fixed point free quintics $Y$ over $\bbC$ is $9$ (see \cite[Example 3]{BoM}). By Remark \ref{rem:5.4} we know that if the universal cover of a Godeaux surface has Picard number $9$, then it is a quintic. Hence, Proposition \ref{prp:4.4} shows that it suffices to find a single smooth, $\bbZ/5$-invariant and fixed point free quintic $Y$ with $\rho(Y)=9$, so that if $X\coloneqq Y/(\bbZ/5)$, then $q^{*}_{X}\colon\br(X)\to\br(Y)$ is injective. But this follows immediately from Theorem \ref{thm:1.3}, by considering a generic pencil of $\bbZ/5$-invariant quintics, such that some member has $5$ triple points as its only singularities (see \cite[Appendix 3, Lemma 1]{persson}).\end{proof}\section{Application II: Enriques Surface}\label{sec:6} Let $k$ be an algebraically closed field of characteristic $0$ and let $R\coloneqq k[[t]]$. By \cite[Corollary 6.2]{mor}, up to birational equivalence, there are exactly three types of strictly semi-stable degenerations of Enriques $\sX\to\Spec R$, such that the canonical divisor $K_{\sX}$ is $2$-torsion. We refer to them as $(\text{i}a)$ $\sX_{0}$ is a smooth Enriques' surface, $(\text{ii}a)$ $\sX_{0}$ is an elliptic chain with one rational component and $(\text{iii}a)$ $\sX_{0}$ is rationally polygonal and its dual graph $\Gamma$ is a triangulation of $\bbR\bbP^{2}$ (see \cite[Corollary 6.2]{mor}).\par The corresponding double cover $\sY\to\Spec R$ is a strictly semi-stable degeneration of a $K3$ surface, with $(\text{i}a)$ $\sY_{0}$ is a $K3$ surface, $(\text{ii}a)$ $\sY_{0}$ is an elliptic chain with two rational components and $(\text{iii}a)$ $\sY_{0}$ is rationally polygonal and its dual graph $\Gamma$ is a triangulation of $S^{2}$ (see \cite[Theorem 6.1]{mor}).\par Let $\sM$ denote the coarse moduli space of Enriques surfaces. Beauville showed that Enriques surfaces $X$ with $q^{*}_{X}=0$ form an infinite, countable union of non-empty hypersurfaces in $\sM$ (see \cite[Corollary 6.5]{bea}). These hypersurfaces have been explicitly described in \cite{bea}, using lattice theory (see \cite[Proposition 6.2]{bea}). In what follows, $\sH\subset\sM$ denotes the union of these hypersurfaces. To the author's knowledge the following result is new:\begin{theorem}\label{thm:6.2} Let $R$ be a discrete valuation ring with fraction field $K$ of characteristic $0$ and algebraically closed residue field $k$ of characteristic $\neq2$. Let $\sX\to\Spec R$ be a type $(\text{ii}a)$ degeneration of an Enriques surface. Consider its canonical cover $q\colon\sY\to\sX$. Then $q^{*}_{\bar{\eta}}\colon\br(\sX_{\bar{\eta}})[2]\to\br(\sY_{\bar{\eta}})[2]$ is non-zero.\end{theorem}\begin{proof} It is clear that a type $(\text{ii}a)$ degeneration of an Enriques surface and its double covering $q\colon\sY\to\sX$ satisfy the conditions in Theorem \ref{thm:1.4}. Hence, it suffices to show that the generator $\alpha\in\br(\sX_{\bar{\eta}})\cong\bbZ/2$ cannot be lifted to $\br(\sX)\{2\}$. This holds true due to Lemma \ref{lem:6.3} (cf. \cite[Remark 7.4]{Stef}).\end{proof}\begin{lemma}\label{lem:6.3} Let $R$ be a discrete valuation ring with fraction field $K$ of characteristic $0$ and algebraically closed residue field $k$ of characteristic $\neq2$. Let $\sX\to\Spec R$ be a type $(\text{ii}a)$ degeneration of an Enriques surface. Then the map $\br(\sX)\to\br(\sX_{\bar{\eta}})$ is trivial.\end{lemma}\begin{proof} The proof is more or less the same as in Lemma \ref{lem:5.2}. We may replace $R$ by its henselization $R^{\text{h}}$. We number the components $\sX_{0,i},$ such that $C_{i,i+1}\coloneqq\sX_{0,i}\cap\sX_{0,i+1}\neq\emptyset$ for all $i$. We also set $C_{\sX_{0,i}}\coloneqq\sX_{0,i}\cap\sX^{\sing}_{0}$.\par An Enriques surface has invariants $p_{q}=q=0$. Hence, Lemma \ref{Lem:4.1} yields the isomorphism $\br(\sX)\cong\br(\sX_{0})_{\tor}$ and it suffices to prove $\br(\sX_{0})=0$. The exact sequence \eqref{5.2} holds here as well and so we only need to show that the map $r$ is surjective. Recall that the special fibre $\sX_{0}$ is an elliptic chain with rational end-component $\sX_{0,n},$ and thus the relation \eqref{5.3} is also true. It is therefore enough to prove that $([x],0,\ldots,0)\in\im(r)$ for a single point $x\in C_{1,2}$, as $C_{1,2}$ is an \'etale double cover onto the base of the elliptic ruled component $\sX_{0,1}$ (cf. the argument in Lemma \ref{lem:5.2}).\par We claim that there is at least one component $\sX_{0,i}$, which is not a minimal surface. The canonical bundle formula gives $K_{\sX_{0,i}}=K_{\sX}\arrowvert_{\sX_{0,i}}-C_{\sX_{0,i}}$. Since $K_{\sX}$ is $2$-torsion in $\Pic(\sX)$, we obtain $K_{\sX_{0,i}}+C_{\sX_{0,i}}\equiv0$. For a contradiction, assume now that all components of the special fibre $\sX_{0}$ are minimal. As $\sX_{0,i}$ are ruled for $i=1,\ldots,n-1$, we get $K^{2}_{\sX_{0,i}}=C^{2}_{\sX_{0,i}}=0$. Recall that $\sX_{0,n}$ is a rational surface and so $K^{2}_{\sX_{0,n}}=C^{2}_{n-1,n}=8$ or $9$. The first relation yields $C^{2}_{n-1,n}=0$, which contradicts the last one.\par Pick a component $\sX_{0,i}$, which is not minimal. For a $(-1)$-rational curve $D\subset\sX_{0,i}$, the adjunction formula $(K_{\sX_{0,i}}+D).D=-2$ yields $D.C_{\sX_{0,i}}=1$. Therefore, $r(0,\ldots,[D],0,\dots,0)=(0,\ldots,[x],\ldots,0)$ for some point $x\in C_{j-1,j}$, where $j=i$ or $i+1$. By \eqref{5.3} we also deduce a point $x\in C_{1,2}$, such that $([x],\ldots,0)\in\im(r),$ finishing the proof.\end{proof}
As an immediate consequence of Theorem \ref{thm:6.2}, we obtain:
\begin{corollary}\label{cor:6.5} Any member $X$ in $\sH$ cannot be degenerated to a singular Enriques surface, such that after applying semi-stable reduction the special fibre becomes of type $(\text{ii}a)$.\end{corollary}\section{Application III: Quotients of Products of Curves} As a last application of Theorem \ref{thm:1.4}, we prove the following.\begin{theorem}\label{thm:7.1}Let $E$ be a smooth elliptic curve and $C$ a smooth projective curve over $\bbC$. Let $\psi$ be an automorphism of $C$ of finite order $d,$ with $\Fix(\psi)\neq\emptyset$. Assume $\Hom(E,\Jac(C))=\{0\}$. Then there exists $0\neq\tau\in E[d]$ with the following property: If $Y\coloneqq E\times C$ and $q_{\tau}\colon Y\to X\coloneqq E\times C/(\bbZ/d)$ denotes the quotient by the diagonal action \[(x,y)\mapsto (x+\tau,\psi(y)),\tag{7.1}\label{7.1}\] then for any $\alpha\in\br(X)\setminus\br(X)_{\div},$ the pullback $q^{*}_{\tau}\alpha\in\br(Y)$ is non-zero.\end{theorem} We follow the same strategy as in the previous sections. We first review degenerations of elliptic curves.\subsection{Cycle degenerations}\label{subsec:7.1}Let $R$ be a discrete valuation ring with residue field $k$ of any characteristic and fraction field $K$. Let $E$ be an elliptic curve over $K$. It is well-known that we can always find a finite extension $R\subset\tilde{R}$ of discrete valuation rings, such that the fraction field $\tilde{K}$ of $\tilde{R}$ is a finite seperable field extension of $K$ and the minimal model $\sE\to\Spec\tilde{R}$ of the base change $E_{\tilde{K}}$ is semi-stable over $\tilde{R}$ (see \cite[\href{https://stacks.math.columbia.edu/tag/0CDM}{Tag 0CDM}]{stacks-project}). There are two particular possibilities for the special fibre $E_{0}\coloneqq\sE_{0}$:\begin{enumerate}
    \item [$(I_{0})$] $E_{0}$ is a smooth elliptic curve.
    \item [$(I_{\nu})$] $E_{0}$ is a N\'eron $\nu$-gon for some integer $\nu\geq 2$, i.e. $E_{0}$ is isomorphic to the quotient of $\bbP^{1}_{\tilde{k}}\times(\bbZ/\nu)$ obtained by identifying the $\infty$-section of the $i$-th copy of $\bbP^{1}$ with the $0$-section of $(i+1)$-st.  
\end{enumerate}\par The smooth locus $\sE^{\sm}$ is a commutative smooth group scheme over $\tilde{R}$, which is the N\'eron model of $E_{\tilde{K}}$. In case $E_{0}$ is a N\'eron $\nu$-gon, the smooth locus of the special fibre is the affine group scheme $\bbG_{m,\tilde{k}}\times\bbZ/\nu$. If furthermore, $\nu$ is invertible in $\tilde{R}$, then the group scheme $\sE^{\sm}[\nu]\coloneqq\ker(\sE^{\sm}\overset{\times\nu}\to\sE^{\sm})$ is \'etale locally isomorphic to $(\bbZ/\nu)^{2}$ (see \cite[Proposition 20.7]{ab}) and thus, after an unramified base change of discrete valuation rings, we may assume that \[\sE^{\sm}[\nu]\cong(\bbZ/\nu)^{2}.\tag{7.2}\label{7.2}\]\par The natural action of $\sE^{\sm}[\nu]$ on $\sE^{\sm}$ extends to $\sE$ (see \cite[Proposition 9.3.13]{lq}). The action on the special fibre $E_{0}$ can be described as follows: The first direct summand of $E^{\sm}_{0}[\nu]\cong\mu_{\nu}\times\bbZ/\nu$ acts on each component of the special fibre via multiplication by $\nu$-th roots of unity and fixes the two vertices $0$ and $\infty$, while the second summand rotates the components.\par\subsection{Proof of Theorem \ref{thm:7.1}}\label{subsec:7.2} Fix a positive integer $d>1$. Let $\sE\to\Spec R$ be a semi-stable model  of an elliptic curve $E_{K}$, whose special fibre is of type $I_{\kappa d}$ for some $\kappa\geq2$ and $\nu\coloneqq \kappa d$ is invertible in $R$. Up to some unramified base change, we may assume that \eqref{7.2} holds. There are exactly $n_{d}\coloneqq d\varphi(d)$ points $0\neq\tau\in\sE^{\sm}[d]$, such that the quotient map $E_{0}\to E_{0}/\tau$ is trivial over each component of $E_{0}/\tau$. Via the isomorphism $E^{\sm}_{0}[\nu]\cong\mu_{\nu}\times\bbZ/\nu$, they correspond to tuples ($\zeta^{\kappa},i\kappa$), where $\zeta$ is a $d$-th root of unity and $1\leq i\leq\nu$ is an integer prime to $d$.\par Pick any $0\neq\tau\in\sE^{\sm}[d]$ from these $n_{d}$ points. Consider a smooth proper family of curves $\sC\to\Spec R$ and an $R$-automorphism $\psi$ of order $d$ that fixes a section of this family. We let $\bbZ/d$ act diagonally on the product $\sY\coloneqq\sE\times_{R}\sC$ as in \eqref{7.1} and form the quotient \[q_{\tau}\colon\sY\to\sX\coloneqq(\sE\times_{R}\sC)/(\bbZ/d) \label{7.6}. \tag{7.3}\]\par In contrast to the previous two applications, the Brauer group of the special fibre of a cycle degeneration has infinitely many torsion elements.\begin{lemma}\label{lem:7.2} Let $R$ be a discrete valuation ring with fraction field $K$ and residue field $k$. Let $N_{0}$ be a smooth projective variety over $k$. Consider a triple-point free semi-stable degeneration $\sN\to\Spec R$, such that the dual graph $\Gamma$ of $\sN_{0}$ is a cycle with components $\sN_{0,i}\cong\bbP^{1}\times N_{0}$ and each double intersection $\sN_{0,i}\cap\sN_{0,i+1}$ is isomorphic to a fibre of $\bbP^{1}\times N_{0}\to\bbP^{1}$. Then there is a short exact sequence \[0\to\Pic(N_{0})\to\br(\sN_{0})\to\br(N_{0})\to0.\tag{7.4}\label{7.3}\]\end{lemma}\begin{proof} The Mayer-Vietoris sequence \eqref{2.4} gives rise to the following exact sequence \[\bigoplus^{n}_{i=1}\Pic (\bbP^{1}\times N_{0})\overset{\alpha}\longrightarrow\bigoplus^{n}_{i=1}\Pic(N_{0})\longrightarrow\br(\sN_{0})\longrightarrow\bigoplus^{n}_{i=1}\br(\bbP^{1}\times N_{0})\overset{\beta}\longrightarrow\bigoplus^{n}_{i=1}\br(N_{0}).\label{7.7}\tag{7.5}\] We note that $\Pic(\bbP^{1}\times N_{0})\cong\Pic(\bbP^{1})\times\Pic(N_{0})\cong\bbZ[N_{0}]\times \Pic(N_{0})$. Since the restriction of $\sO_{\bbP^{1}\times N_{0}}([N_{0}])$ to $N_{0}$ is the trivial line bundle, we deduce that the cokernel of $\alpha$ coincides with the one of the map $\bigoplus^{n}_{i=1}\Pic(N_{0})\to\bigoplus^{n}_{i=1}\Pic(N_{0}),\ (D_{1},\ldots,D_{n})\mapsto (D_{1}-D_{2},\ldots,D_{n}-D_{1}).$ It is readily checked that the summation map $\bigoplus^{n}_{i=1}\Pic(N_{0})\to\Pic(N_{0}),(D_{i})_{i}\mapsto \Sigma_{i}D_{i}$ yields an isomorphism $\coker(\alpha)\cong \Pic(N_{0})$. On the other hand, recall that the Brauer group of smooth projective varieties is a stably birational invariant (see \cite[Proposition 6.2.9]{cts}). Hence, we have a natural isomorphism $\br(N_{0})\cong\br(N_{0}\times\bbP^{1})$ (see \cite[Corollary 6.2.11]{cts}). Under the above identification the map $\beta$ takes the form $(\alpha_{1},\ldots,\alpha_{n})\mapsto(\alpha_{1}-\alpha_{2},\ldots,\alpha_{n}-\alpha_{1})$ and so its kernel is isomorphic to $\br(N_{0})$. The exact sequence \eqref{7.3} thus follows from \eqref{7.7}.\end{proof}\begin{lemma}\label{lem:7.3} Let $R$ be a strictly henselian discrete valuation ring. Let $q_{\tau}\colon\sY\to\sX\coloneqq(\sE\times_{R}\sC)/(\bbZ/d)$ be the quotient map in \eqref{7.6}. For any prime $\ell$ invertible in $R$, restriction yields isomorphisms  \[\br(\sY)\{\ell\}\overset{\cong}\to\br(\sY_{0})\{\ell\}\tag{7.6}\label{7.4}\]\[\br(\sX)\{\ell\}\overset{\cong}\to\br(\sX_{0})\{\ell\}.\tag{7.7}\label{7.5}\]\end{lemma}\begin{proof} Note that the sufficient conditions of Lemmas \ref{Lem:4.1} and \ref{Lem:4.3} are not true here and so we cannot get \eqref{7.4}, \eqref{7.5} for free. Instead, we proceed as follows. The Mayer-Vietoris sequence \eqref{2.4} gives rise to the following short exact sequences:\[1\to k^{*}\to\Pic (\sY_{0})\to (\bigoplus^{\nu}_{i=1}\bbZ[C_{0}])\times \Pic(C_{0})\to 1\]\[1\to k^{*}\to\Pic (\sX_{0})\to (\bigoplus^{\kappa}_{i=1}\bbZ[C_{0}])\times\Pic(C_{0})\to 1.\] Here, we think of $\Pic(C_{0})$ as a subgroup of $\bigoplus\Pic(\bbP^{1}\times C_{0})$ via the diagonal. Since $R$ is strictly henselian, the residue field $k$ is seperably closed and tensoring with $\bbZ/\ell^{r}$ gives isomorphisms \[\Pic (\sY_{0})\otimes\bbZ/\ell^{r}\cong (\bigoplus^{\nu}_{i=1}(\bbZ/\ell^{r})[C_{0}])\times (\Pic(C_{0})\otimes\bbZ/\ell^{r})\] \[\Pic (\sX_{0})\otimes\bbZ/\ell^{r}\cong (\bigoplus^{\kappa}_{i=1}(\bbZ/\ell^{r})[C_{0}])\times (\Pic(C_{0})\otimes\bbZ/\ell^{r}).\]\par By \eqref{4.13} it suffices to prove \[\frac{\Pic(\sY_{0})}{\Pic(\sY)}\otimes\bbQ_{\ell}/\bbZ_{\ell}=0\ \text{and}\ \frac{\Pic(\sX_{0})}{\Pic(\sX)}\otimes\bbQ_{\ell}/\bbZ_{\ell}=0.\]\par Using the group scheme structure of $\sE^{\sm}$, we find $\nu$ sections $s_{i}\colon \Spec R\to\sE^{\sm}$, one for each component of the special fibre. We set $D_{i}\coloneqq\im(s_{i}\times\id_{\sC})\subset\sY$ and $F_{i}\coloneqq q(D_{i})\subset\sX$ for all $i$. Clearly, the restrictions $D_{i}|_{\sY_{0}}\in\Pic(\sY_{0})$ (resp. $F_{i}|_{\sX_{0}}\in\Pic(\sX_{0})$) are generators in $\bigoplus^{\nu}_{i=1}(\bbZ/\ell^{r})[C_{0}]$ (resp. $\bigoplus^{\kappa}_{i=1}(\bbZ/\ell^{r})[C_{0}]$).\par Recall that $\Pic(C_{0})\otimes\bbZ/\ell^{r}\cong\bbZ/\ell^{r}$ is generated by any closed point. Pick a fixed point $x\in\sC(R)$ of $\psi$ and consider the prime divisors $D\coloneqq\im(\id_{\sE}\times x)\subset \sY$ and the quotient $F\coloneqq D/(\bbZ/d)\subset \sX$. Then the restrictions $D|_{\sY_{0}}\in\Pic(\sY_{0})$ and $F|_{\sX_{0}}\in\Pic(\sX_{0})$ yield generators of $\Pic(C_{0})\otimes\bbZ/\ell^{r}$, finishing the proof.\end{proof}\begin{proposition}\label{prp:7.4}Let $R$ be a discrete valuation ring with fraction field $K$ and algebraically closed residue field $k$. Let $q_{\tau}\colon\sY\to\sX\coloneqq(\sE\times_{R}\sC)/(\bbZ/d)$ be the quotient map in \eqref{7.6}. Let $\ell$ be a prime invertible in $R$. For any class $\alpha\in\br(\sX_{\bar{\eta}})\{\ell\}\setminus\br(\sX_{\bar{\eta}})_{\div}\{\ell\},$ the pull-back to $\sY_{\bar{\eta}}$ is non-zero: $0\neq q^{*}_{\bar{\eta}}\alpha\in\br(\sY_{\bar{\eta}})$.\end{proposition}\begin{proof} We may assume that $\alpha\in\br(\sX_{\bar{\eta}})\{\ell\}\setminus\br(\sX_{\bar{\eta}})_{\div}\{\ell\},$ for some prime factor $\ell$ of $d$. A sufficient condition for the pull-back $q^{*}_{\bar{\eta}}\alpha$ to be non-zero is to show that up to every finite base change of discrete valuation rings followed by a resolution of the total space (see \cite{Har}), $\alpha$ cannot be lifted to a class in $\br(\sX)\{\ell\}$ (see Theorem \ref{thm:3.1}).\par As the Brauer group of a curve over an algebraically closed field is zero (see \cite[Theorem 5.6.1]{cts}), Lemma \ref{lem:7.2} implies $\br(\sX_{0})\{\ell\}\cong\Pic(C_{0})\{\ell\}$. By passing to the henselization of the base, Lemma \ref{lem:7.3} yields $\br(\sX)\{\ell\}\cong\Pic(C_{0})\{\ell\}$. The group $\Pic(C_{0})\{\ell\}\cong(\bbQ_{\ell}/\bbZ_{\ell})^{2g}$ is divisible, whereas $\br(\sX_{\bar{\eta}})\{\ell\}/\br(\sX_{\bar{\eta}})_{\div}\{\ell\}\cong H^{3}(\sX_{\bar{\eta}},\bbZ_{\ell}(1))_{\tor}$ is finite (see \cite[Proposition 5.2.9]{cts}). These observations certainly imply that $\br(\sX)\{\ell\}\to\br(\sX_{\bar{\eta}})\{\ell\}/\br(\sX_{\bar{\eta}})_{\div}\{\ell\}$ is always the zero map, as claimed.\end{proof} Finally, with the help of the above preparation, we are able to prove Theorem \ref{thm:7.1}.\begin{proof}[Proof of Theorem \ref{thm:7.1}] Let $C$ be a smooth complex projective curve and let $\psi\in\Aut(C)$ be an automorphism of order $d$, with $\Fix(\psi)\neq\emptyset$. We claim that Proposition \ref{prp:7.4} yields an example of an elliptic curve $E$ over $\bbC$ with $\Hom(E,\Jac(C))=\{0\}$ and a point $0\neq\tau\in E[d]$, for which the conclusion of Theorem \ref{thm:7.1} holds true.\par To see this, we may choose a countable algebraically closed field $k\subset\bbC$, such that $C=C_{k}\times_{k}\bbC$. Let $\sE$ be a cycle degeneration over the local ring of a smooth pointed $k$-curve $(B,0)$ with special fibre of type $I_{\kappa d}$, $\kappa\geq 2$. We may assume that \eqref{7.2} holds. Set $\sC\coloneqq C_{k}\times_{k}\sO_{B,0}$ and consider the quotient map $q_{\tau}\colon\sY\to\sX\coloneqq(\sC\times_{\sO_{B,0}}\sE)/(\bbZ/d)$ (from \eqref{7.6}) for some point $0\neq\tau\in\sE^{\sm}[d]$ that restricts via the isomorphism $E^{\sm}_{0}[\nu]\cong \mu_{\nu}\times\bbZ/\nu$ to a tuple of the form $(\zeta^{\kappa},i\kappa)$, where the integer $1\leq i\leq \nu\coloneqq\kappa d$ is prime to $d$ and $\zeta$ is any $d$-th root of unity.\par Proposition \ref{prp:7.4} implies that $q^{*}_{\bar{\eta}}\alpha\neq0$ for all $\alpha\in\br(\sX_{\bar{\eta}})\setminus\br(\sX_{\bar{\eta}})_{\div}$. We pick an embedding $k(B)\subset\bbC$ that respects the given one $k\subset\bbC$ and perform the base change. This yields an example $q_{\tau}\colon E\times C\to(E\times C)/(\bbZ/d)$ defined over $\bbC$ that satisfies the conclusion of Theorem \ref{thm:7.1}. Note that $\Hom(E,\Jac(C))=\{0\}$, because $\sE_{\eta}$ has multiplicative reduction at $0\in B$, whereas $\Jac(C_{\eta})$ has abelian reduction everywhere in $B$.\par Next, we proceed as in the proof of Theorem \ref{thm:1.1}. We consider the Legendre family of elliptic curves $p\colon\sF\subset\bbP_{\bbC}^{2}\times U\to U$, $U\coloneqq\bbA^{1}_{\bbC}-\{0,1\}$, whose fibres $\sF_{\lambda}$ are defined by the affine equation \[y^{2}=x(x-1)(x-\lambda)\] and recall that every elliptic curve is isomorphic to some fibre of this family. We regard $\sF\to U$ as an abelian scheme, with the identity section given by the point $(0:1:0)$.\par Pick $\lambda_{0}\in U$, such that $E=\sF_{\lambda_{0}}$. Up to some finite base change, we may assume that $\tau\in E[d]$ lifts to a section of the family $\sF\to U$. We let the group $\bbZ/d$ act on $\sF$ via the translation $\tau$ and on $C$ via the automorphism $\psi$ and consider the quotient map \[q_{\sF}\colon \sF\times C\to(\sF\times C)/(\bbZ/d).\]For $\lambda\in U$, we set $q_{\lambda}\coloneqq q_{\sF_{\lambda}}$, $Y_{\lambda}\coloneqq \sF_{\lambda}\times C$ and $X_{\lambda}\coloneqq (\sF_{\lambda}\times C)/(\bbZ/d)$.\par Choose $\lambda_{1}\in U$, such that $\Hom(\sF_{\lambda_{1}},\Jac(C))=\{0\}$. Then the maps $sp_{\bar{\eta},\lambda_{i}}\colon\br(Y_{\bar{\eta}})\to\br(Y_{\lambda_{i}})$ are isomorphisms for $i=0,1$ (see Lemma \ref{lem:4.4}). By the smooth proper base change theorem, the quotient of the Brauer group of a smooth projective variety over an algebraically closed field of characteristic $0$ by its maximal divisible subgroup is invariant in smooth proper families (see \cite[Proposition 5.2.9]{cts}). Thus, the surjectivity of $sp_{\bar{\eta},\lambda_{i}}\colon\br(X_{\bar{\eta}})\onto\br(X_{\lambda_{i}})$ yields an isomorphism \[\br(X_{\bar{\eta}})/\br(X_{\bar{\eta}})_{\div}\cong\br(X_{\lambda_{i}})/\br(X_{\lambda_{i}})_{\div}.\]Since specialization is compatible with pullbacks (see Lemma \ref{lem:4.4}) the claim follows by comparing the three pull-backs $q^{*}_{\bar{\eta}}$, $q^{*}_{\lambda_{0}}$ and $q^{*}_{\lambda_{1}}$ via the specialization maps. The proof of Theorem \ref{thm:7.1} is complete.\end{proof}\begin{remark}\label{rem:5.9} Assume that the pair $(C,\psi)$ satisfies $C/\psi\cong\bbP^{1}$. Then $\br((E\times C)/(\bbZ/d))_{\div}=0$ and so Theorem \ref{thm:7.1} says that the pull-back map $q_{\tau}^{*}\colon\br((E\times C)/(\bbZ/d))\to\br(E\times C)$ is injective for some $0\neq\tau\in E[d]$, if $\Hom(E,\Jac(C))=\{0\}$.\par The conclusion of Theorem \ref{thm:7.1} is known for Bi-elliptic surfaces and their canonical covers, without any restriction on the choice of the torsion point $\tau\in E[d]$ (see \cite[Theorem B]{bie}).\end{remark}\begin{example} It is readily checked that the datum $(C,\psi)$ can be replaced by any pair $(V,\psi)$, where $V$ is a smooth complex projective variety and $\psi$ is an automorphism of $V$ of order $d$ satisfying the following two properties:\begin{enumerate}
\item\label{1'} For all prime factors $\ell$ of $d,$ the group $\NS(V)\otimes\bbQ_{\ell}/\bbZ_{\ell}$ is generated by prime divisors that are invariant under $\psi$.
\item\label{2'} $\bigoplus_{\ell|d}\Pic(V)\{\ell\}=\bigoplus_{\ell|d}\Pic^{0}(V)\{\ell\}$.
\end{enumerate} Item \eqref{1'} is needed to ensure $\bigoplus_{\ell|d}\br(\sX)\{\ell\}\cong\bigoplus_{\ell|d}\br(\sX_{0})\{\ell\}$, where $\sX=(\sE\times V)/(\bbZ/d)$ (cf. Proof of Lemma \ref{lem:7.3}). In the proof of Proposition \ref{prp:7.4} item \eqref{2'} was the key to show that classes $\alpha\in\br(\sX_{\bar{\eta}})\setminus\br(\sX_{\bar{\eta}})_{\div}$ do not lift to the total space $\sX$. However, the group $\br(\sX_{0})_{\tor}$ may not be divisible (see \eqref{7.3}) and so the same argument might not work here. We prove Proposition \ref{prp:7.4} for the quotient map $q_{\tau}\colon\sY\coloneqq \sE\times V\to\sX\coloneqq(\sE\times V)/(\bbZ/d)$, where $0\neq \tau\in\sE[d]$ is chosen so that $E_{0}\to E_{0}/\tau$ is generically trivial: For a contradiction, assume that there exists a non-zero $d$-torsion class $\alpha\in\br(\sX_{\bar{\eta}})\setminus\br(\sX_{\bar{\eta}})_{\div}$, such that $q^{*}_{\bar{\eta}}(\alpha)=0\in \br(\sY_{\bar{\eta}})$. Then after a finite base change of discrete valuation rings folllowed by a resolution of the total space, we may assume that $\alpha$ lifts to a class $\tilde{\alpha}\in\br(\sX)[d]$ and $q^{*}\tilde{\alpha}=0\in\br(\sY)[d]$ (see Theorem \ref{thm:3.1}). But by item \eqref{2'} via the last homomorphism in \eqref{7.3} the restriction $0\neq\tilde{\alpha}|_{\sX_{0}}$ must map to a non-zero element in $\br(V)$. Finally, the commutativity of the following diagram\[\begin{tikzcd}
\br(\sX) \arrow[r,] \arrow[d, "q^{*}"] & \br(\sX_{0}) \arrow[r,] \arrow[d, "q_{0}^{*}"] & \br(V) \arrow[d, "\{(\psi^{i})^{*}\}^{d}_{i=1}"] \\
\br(\sY) \arrow[r]                            & \br(\sY_{0}) \arrow[r]                                & \bigoplus^{d}_{i=1}\br(V),                  
\end{tikzcd}\]yields $q^{*}\tilde{\alpha}\neq 0$, which contradicts $q^{*}\tilde{\alpha}=0$.\end{example}  
\section*{Acknowledgements} I am grateful to my supervisor Stefan Schreieder for many helpful comments and discussions concerning this work. This project has received funding from the European Research Council (ERC) under the European Union's Horizon 2020 research and innovation programme under grant agreement No 948066.
\printbibliography
\address{INSTITUTE OF ALGEBRAIC GEOMETRY, LEIBNIZ UNIVERSITY HANNOVER, WELFENGARTEN 1, 30167 HANNOVER, GERMANY.}\\
\textit{Email address}:\ \email{alexandrou@math.uni-hannover.de}
\end{document}